\documentclass[11pt]{article}
\usepackage{amsmath}
\usepackage{amsthm}
\usepackage{amsfonts}
\usepackage{amssymb}
\usepackage{mathtools}
\usepackage[english]{babel}
\usepackage{combelow}
\usepackage{hyperref}

\newtheorem{thm}{Theorem}[section]
\newtheorem{lem}[thm]{Lemma}

\theoremstyle{definition}

\newcommand{\BQ}{{\mathbb{Q}}}

\newcommand{\ceil}[1]{\left\lceil#1\right\rceil}
\newcommand{\gp}{\mathfrak{p}}

\let\lm=\lambda
\let\Lm=\Lambda

\let\abs=\envert

\theoremstyle{remark}

\begin{document}
\title{A generalization of the Ramanujan-Nagell equation
}
\author{Tomohiro Yamada}
\date{}
\maketitle

\begin{abstract}
We shall show that, for any positive integer $D>0$ and any primes $p_1, p_2$ not dividing $D$,
the diophantine equation $x^2+D=2^s p_1^k p_2^l$ has at most $63$ integer solutions $(x, k, l, s)$
with $x, k, l\geq 0$ and $s\in \{0, 2\}$.
\end{abstract}
\section{Introduction}\label{intro}

It is known that the equation $x^2+7=2^n$ has five solutions,
as conjectured by Ramanujan and shown by Nagell \cite{Nag2} and other authors.
According to this history, this diophantine equation has been called
the Ramanujan-Nagell equation and several authors have studied various analogues.

Ap\'{e}ry \cite{Ape} showed that, for each integer $D>0$ and prime $p$,
the equation $x^2+D=p^n$ has at most two solutions unless $(p, D)=(2, 7)$ and,
for any odd prime $p$, the equation $x^2+D=4p^n$, which is equivalent to $y^2+y+(D+1)/4=p^n$ with $y$ odd,
also has at most two solutions.
Beukers \cite{Beu1} showed that, if $D>0$ and $x^2+D=2^n$ has two solutions,
then $D=23$ or $D=2^k-1$ for some $k>3$ and also gave an effective upper bound:
if $w=x^2+D=2^n$ with $D\neq 0$, then $w<2^{435}\abs{D}^{10}$.

Further generalizations have been made by Le \cite{Le2}\cite{Le3}\cite{Le4}, Skinner \cite{Ski}
and Bender and Herzberg \cite{BH} to prove that,
for any given integers $A, B, s, p$ with $\gcd(A, B)=1, s\in \{0, 2\}$
and $p$ prime, $Ax^2+B=2^s p^n$ has at most two solutions except
$2x^2+1=3^k, 3x^2+5=2^k, x^2+11=4\times 3^k, x^2+19=4\times 5^k$ with three solutions and
the Ramanujan-Nagell one $x^2+7=2^k$ with five solutions.

Bender and Herzberg \cite{BH} also found some necessary conditions for 
the equation $D_1x^2+D_2=2^s a^n$ with $D_1>0, D_2>0, \gcd(D_1, D_2)=\gcd(D_1D_2, k)=1, s\in\{0, 2\}$
to have more than $2^{\omega(a)}$ solutions.
With the aid of the primitive divisor theorem of Bilu, Hanrot and Voutier \cite{BHV} concerning
Lucas and Lehmer sequences, Bugeaud and Shorey\cite{BuSh} determined all cases
$D_1x^2+D_2=2^m a^n$ with $D_1>0, D_2>0, \gcd(D_1, D_2)=\gcd(D_1D_2, k)=1, m\in\{0, 1, 2\}$
has more than $2^{\omega(a)-1}$ solutions, although they erroneously refer
to $2x^2+1=3^n$ as it has just two solutions $n=1, 2$, which in fact has exactly
three solutions $n=1, 2, 5$, as pointed out by Leu and Li \cite{LL}
(this fact immediately follows from Ljunggren's result \cite{Lju}
since $2x^2+1=3^n$ is equivalent to $(3^n-1)/2=x^2$).

We note that it is implicit in Le \cite{Le1} that, if $D_1>3$, then
$D_1 x^2+1=p^n$ has at most one solution except $(D_1, p)=(7, 2)$.
But it is erroneously cited in another work of Le \cite{Le5}, stating that
$D_1 x^2+1=p^n$ has at most one solutions for each $D_1\geq 1$ and odd prime $p$.
This may have caused the failure in \cite{BuSh} mentioned above.

Le \cite{Le11} studied another generalized Ramanujan-Nagell equation
$x^2+D^m=p^n$ with $m, n, x>0$, $p$ a prime not dividing $D$ to show that this equation
has at most two solutions except for some special cases.
Further studies by Bugeaud \cite{Bu1} and Yuan and Hu\cite{YH} concluded that this equation has
at most two solutions except for $(D, p)=(7, 2), (2, 5)$ and $(4, 5)$,
in which cases, this equation has, respectively, exactly six, three and three.
Hu and Le \cite{HL} showed that,
for integers $D_1, D_2>1$ and a prime $p$ not dividing $D_1 D_2$, the equation
$D_1 x^2+D_2^m=p^n, x, m, n>0$ has at most two solutions
except for $(D_1, D_2, p)=(2, 7, 13), (10, 3, 13), (10, 3, 37)$ and
$((3^{2l}-1)/a^2, 3, 4\times 3^{2l-1}-1)$ with $a, l\geq 1$,
in which cases this equation has exactly three solutions.

The diophantine equation $x^2+D=y^n$ with only $D$ given also has been studied.
Lebesgue \cite{Leb} solved this equation for $D=1$, Nagell solved for $D=3$
and Cohn \cite{Coh} solved for many values of $D$.
By the theorem of Shorey, van der Poorten, Tijdeman and Schinzel \cite{SPTS},
we have $x, y, n\leq C$ with an effectively computable constant $C$ depending only on $D$.
Combining a modular approach developed by Taylor and Wiles \cite{TW}\cite{Wil} and Bennett and Skinner \cite{BeSk} and other methods, Bugeaud, Mignotte and Siksek \cite{BMS}
solved $x^2+D=y^n$ in $(x, y, n)$ with $n\geq 3$ for each $1\leq D\leq 100$.
Furthermore, Le \cite{Le12} showed that if $x^2+2^m=y^n$ with $m, x>0$, $n>2$ and $y$ odd, then
$(x, m, y, n)=(5, 3, 1, 3), (7, 3, 5, 4)$ or $(11, 5, 2, 3)$.
Pink \cite{Pin} solved $x^2+D=y^n, n\geq 3, \gcd(x, y)=1$ for $D=2^a 3^b 5^c 7^d$
except the case $D\equiv 7\pmod{8}$ and $y$ is even.
A brief survey on further results to such equations is given by B\'{e}rczes and Pink \cite{BP}.
More recently, Godinho, Diego Marques and Alain Togb\'{e} \cite{GMT} solved
$x^2+D=y^n, n\geq 3, \gcd(x, y)=1$ for $D=2^a 3^b 17^c$ and $D=2^a 13^b 17^c$.

In this paper, we shall study another generalization of the Ramanujan-Nagell equation
\begin{equation}\label{eq0}
x^2+D=2^s p_1^k p_2^l
\end{equation}
with $s\in \{0, 2\}$.

Evertse \cite{Eve} showed that,
for every nonzero integer $D$ and $r$ prime numbers $p_1, p_2, \ldots, p_r$,
$x^2+D=p_1^{k_1} p_2^{k_2} \cdots p_r^{k_r}$ has at most $3\times 7^{4r+6}$ solutions.
Hence, (\ref{eq0}) has at most $3\times 7^{14}$ solutions
for any given $D, p_1, p_2$.
The purpose of this paper is to improve this upper bound for the number of solutions of (\ref{eq0}).
\begin{thm}\label{thm1}
For every positive integer $D$ and primes $p_1, p_2$,
(\ref{eq0}) has at most $63$ integral solutions $(x, s, k, l)$ with $k, l\geq 0, s\in \{0, 2\}$.
\end{thm}

It seems that we cannot use the primitive divisor theory for such types of equations.
Instead, we shall use Beukers' method.  However, we need more complicated argument than
Beukers' original argument in \cite{Beu1}.

Let $P(x)=x^2+D$.  Hence, (\ref{eq0}) can be rewritten as $P(x)=2^s p_1^k p_2^l$.
In order to extend Beukers' argument for (\ref{eq0}), we shall divide the set of solutions
of this equation.
Let $S(\alpha, \alpha+\delta, X, Y)=S_{P(x)}(\alpha, \alpha+\delta, X, Y)$ be the set of solutions
of the equation $P(x)=2^s p_1^k p_2^l$ with
$X\leq P(x)<Y, s\in\{0, 2\}$ and $(p_1^k p_2^l)^\alpha\leq p_1^k\leq(p_1^k p_2^l)^{\alpha+\delta}$
and we write $S(\alpha, \alpha+\delta)=S(\alpha, \alpha+\delta, 0, \infty)$ for brevity.
Moreover, for $u, v\pmod{2}$,
let $S(\alpha, \alpha+\delta, X, Y; u, v)=S_{P(x)}(\alpha, \alpha+\delta, X, Y; u, v)$ be the set of solutions
$x^2+D=2^s p^k q^l\in S(\alpha, \alpha+\delta, X, Y)$ with
$k\equiv u\pmod{2}, l\equiv v\pmod{2}$
and $S(\alpha, \alpha+\delta; u, v)=S(\alpha, \alpha+\delta, 0, \infty; u, v)$.
Finally, let us write $S^{(j)}(X, Y; u, v)=S(j/4, (j+1)/4, X, Y; u, v)$ for $j=0, 1, 2, 3$
and so on.

Now, we shall state our result in more detail.
\begin{thm}
Let $Y=4883601$ and $W$ be the constant defined in Lemma \ref{lm21} with $\delta=1/4$ and $\delta_1=0.04377667$.
Moreover, let $y_1$ be the smallest solution of (\ref{eq0}).
For every positive integer $D$ and primes $p_1<p_2$, we have
\begin{itemize}
\item[(i)] Each $S^{(j)}(W, \infty; u, v)$ contains at most three solutions for $j=1, 2$ and two solutions for $j=0, 3$.  Hence, there exist at most $30$ solutions with $x^2+D\geq W$.
\item[(ii)] If $D\geq Y$ or $y_1\geq Y$, then
$S^{(j)}(y_1, W)$ contains at most nine solutions for $j=1, 2$ and five solutions for $j=0, 3$.
Hence, there exist at most $28$ solutions with $x^2+D<W$.
\item[(iii)] If $D, y_1, p_2<Y$, then there exist at most $29$ solutions with $x^2+D<W$.
\item[(iv)] If $D, y_1<Y<p_2$, then
$S^{(j)}(Y, W)$ contains at most nine solutions for $j=1, 2$ and five solutions for $j=0, 3$.
Hence, there exist at most $28$ solutions with $Y\leq x^2+D<W$.
Moreover, there exist at most $5$ solutions with $x^2+D<Y$.
\end{itemize}
\end{thm}

In the next section, we prove a weaker gap principle using only elementary argument using congruences,
which is used to bound the number of middle solutions (and as an auxiliary tool to prove a stronger gap principle in Section 4).
In Section 3, we use Beukers' argument to show that
if we have one large solution $w=x^2+D$ in a class $S^{(j)}(W, \infty; u, v)$,
then other solutions in the same class as $w$ must be bounded by $w$.
Combining an gap argument proved in Section 4, we obtain an upper bound for the number of solutions in each class.
The number of small solutions can be checked by computer search.

\section{An elementary gap argument}
In this section, we shall give the following two gap principles shown by elementary arguments using congruence.

\begin{lem}\label{lm12a}
Let $x_1<x_2$ be two integers such that $y_i=x_i^2+D (i=1, 2)$ belong to the same set $S_{P(x)}(\alpha, \alpha+\delta)$,
where $\alpha, \delta$ are two real numbers satisfying $0\leq \delta<1/4$
and $\alpha=0$ or $0\leq \alpha\leq 1\leq \alpha+\delta$.
Then we have $x_2>\frac{1}{2}(P(x_1)/4)^{3/4}$.
\end{lem}
\begin{proof}
Let $x_1<x_2$ be two integers in $S_{P(x)}(3/4, 1)$.  Then we can easily see that
$P(x_i)\equiv 0\pmod{p_1^{e_i}}$ with $p_1^{e_i}\geq (P(x_1)/4)^{3/4}$.
This implies that $P(x_1)\equiv P(x_2)\equiv 0\pmod{p_1^f}$, where $f=\min\{e_1, e_2\}$.
Hence, we have $x_1+x_2\geq p_1^f\geq (P(x_1)/4)^{3/4}$ and therefore
$x_2>\frac{1}{2}(P(x_1)/4)^{3/4}$.
Similarly, if $x_1<x_2$ are two integers in $S_{P(x)}(0, 1/4)$, then $x_2>\frac{1}{2}(P(x_1)/4)^{3/4}$.
This proves the lemma.
\end{proof}

\begin{lem}\label{lm12b}
Let $x_1<x_2<x_3$ be three integers such that $y_i=x_i^2+D (i=1, 2, 3)$ belong to the same set $S_{P(x)}(\alpha, \alpha+\delta)$
for some $0\leq \alpha\leq 1$ with $0\leq \delta\leq 1/4$.
Then we have $x_3>\frac{1}{2}(P(x_1)/4)^{3/4}$.
\end{lem}
\begin{proof}
For each $i=1, 2, 3$, we have
$P(x_i)\equiv 0\pmod{p_1^f p_2^g}$, where
$f=\ceil{\alpha\log (P(x_1)/4)/\log p_1}$ and
$g=\ceil{\left(\frac{3}{4}-\alpha\right)\log (P(x_1)/4)/\log p_2}$.

We see that the congruent equation
$X^2+D\equiv 0\pmod{p_1^f p_2^g}$ has exactly four distinct solutions
$0<X_1<X_2<X_3<X_4<p_1^f p_2^g$ with $X_1+X_4=X_2+X_3=p_1^f p_2^g$.
Hence, we have $X_3, X_4>\frac{1}{2}p_1^f p_2^g$ and $x_3>\frac{1}{2}p_1^f p_2^g\geq \frac{1}{2}(P(x_1)/4)^{3/4}>\frac{1}{2^{5/2}}x_1^{3/2}$.
\end{proof}

\section{Hypergeometric functions and finiteness results}

Let $F(\alpha, \beta, \gamma, z)$ be the hypergeometric function given by the series
\begin{equation}
1+\frac{\alpha\cdot\beta}{1\cdot\gamma}z+\frac{\alpha(\alpha+1)\beta(\beta+1)}{1\cdot 2\cdot\gamma(\gamma +1)}z^2+\cdots,
\end{equation}
converging for all $\abs{z}<1$ and for $z=1$ if $\gamma>\alpha+\beta$.
Define
$G(z)=G_{n_1, n_2}(z)=F(-\frac{1}{2}-n_2, -n_1, -n, z), H(z)=G_{n_1, n_2}(z)=F(-\frac{1}{2}-n_1, -n_2, -n, z)$
and $E(z)=F(n_2 +1, n_1+\frac{1}{2}, n+2, z)/F(n_2 +1, n_1+\frac{1}{2}, n+2, 1)$
for positive integers $n, n_1, n_2$ with $n=n_1+n_2$ and $n_1\geq n_2$.

We quote some properties from Lemmas 2-4 of \cite{Beu1}:
\begin{enumerate}
\item[(a)]$\abs{G(z)-\sqrt{1-z}H(z)}<z^{n+1}G(1)$,
\item[(b)]$\binom{n}{n_1}G(4z)$ and $\binom{n}{n_1}H(4z)$ are polynomials with integer coefficients of degree $n_1$ and $n_2$
respectively,
\item[(c)]$G(1)<G(z)<G(0)=1$ for $0<z<1$,
\item[(d)]$G(1)=\binom{n}{n_1}^{-1}\prod_{m=1}{n_1}\left(1-\frac{1}{2m}\right)$ and
\item[(e)]$G_{n_1 +1, n_2 +1}(z)H_{n_1, n_2}(z)-G_{n_1, n_2}(z)H_{n_1 +1, n_2 +1}(z)=cz^{n+1}$
for some constant $c\neq 0$.
\end{enumerate}

Now we obtain the following upper bound for solutions of (\ref{eq0}) relative to a given large one.

\begin{lem}\label{lm21}
Let $\alpha, \delta$ and $\delta_1$ be real numbers with $0\leq \alpha<\alpha+\delta\leq 1$
and $0<\delta_1<1/12$ and
$A, B, w, q, s_1, k_1, l_1, s_2, k_2, l_2$ be nonnegative integers such that
both $A^2+D=w=2^{s_1}p_1^{k_1}p_2^{l_1}$ and $B^2+D=q=2^{s_2}p_1^{k_2}p_2^{l_2}$
belong to $S(\alpha, \alpha+\delta; u, v)$ with $B>A$.
Moreover, put $W_1=(2^{772+210\delta}D^{241})^{1/(35(2-3\delta)-(3\delta+1)/2)}, W_2=(2^{22/9+2\delta/3}3^{7/3})^{1/\delta_1}$
and $W=\max\{W_1, W_2\}$.

If $w\geq W$, then $q<4^{70}w^{71}$ or
\begin{equation}\label{eq2990}
q^{1-\frac{1}{2}\left(\frac{5}{3}+\delta+\delta_1\right)}<2^{\frac{31}{9}+s_1+\frac{2}{3}\delta}3^{\frac{16}{3}}Dw^{\frac{19}{6}+\frac{3}{2}\delta-\frac{1}{2}\left(\frac{5}{3}+\delta+\delta_1\right)}.
\end{equation}
\end{lem}
\begin{proof}
Substituting $z=\frac{D}{w}$, we see that
$\sqrt{1-z}=\frac{A}{w^{\frac{1}{2}}}$ and it follows from the property (b) that
\begin{equation}
\binom{n}{n_1}G(z)=\frac{P}{(4w)^{n_1}}\text{ and }\binom{n}{n_1}H(z)=\frac{Q}{(4w)^{n_2}}
\end{equation}
for some integers $P$ and $Q$.

Now the property (a) gives
\begin{equation}
\abs{\frac{P}{(4w)^{n_1}}-\frac{AQ}{w^{\frac{1}{2}}(4w)^{n_2}}}<\binom{n}{n_1}\left(\frac{D}{w}\right)^{n+1}G(1)
\end{equation}
and therefore
\begin{equation}
\abs{1-\frac{AQ}{w^{\frac{1}{2}}(4w)^{n_2-n_1}P}}<\frac{(4w)^{n_1}}{\abs{P}}\binom{n}{n_1}\left(\frac{D}{w}\right)^{n+1}G(1).
\end{equation}

Letting
\begin{equation}
K=\abs{\frac{B}{q^{\frac{1}{2}}}-\frac{AQ}{w^{\frac{1}{2}}(4w)^{n_2-n_1}P}},
\end{equation}
we have
\begin{equation}\label{eq2991}
K<\epsilon+\frac{(4w)^{n_1}}{\abs{P}}\binom{n}{n_1}\left(\frac{D}{w}\right)^{n+1}G(1),
\end{equation}
where
\begin{equation}
\epsilon=\abs{\frac{B}{q^\frac{1}{2}}-1}<\frac{D}{2B^2}.
\end{equation}

Let $\lm$ be the integer such that $(4w)^{\lm-1}<(q/w)^{1/2}\leq (4w)^\lm$
and choose $n_1, n_2$ such that $\frac{2}{3}\lm-\frac{2}{3}\leq n_1\leq\frac{2}{3}\lm +1,
n_2=n_1+\lm$ and $K\neq 0$.  Following the proof of Theorem 1 in \cite{Beu1}, the property (e) allows such choice.  Moreover, we may assume without loss of generality that $q\geq 4^{70}w^{71}$, which yields that $\lm\geq 35$ and $n_1\geq 23$.

Let $R$ be the l.c.m. of $q$ and $w(4w)^{2\lm}$.  Then, since 
$k_1\equiv l_1, k_2\equiv l_2\pmod{2}$ and we have chosen $n_1, n_2$
such that $K\neq 0$, we see that the denominator $K$ must divide $R^{1/2}\abs{P}$.

Since both $w$ and $q$ belong to $S(\alpha, \alpha+\delta; u, v)$,
we have $p_1^{k_2}\leq (q/2^{s_2})^{\alpha+\delta}\leq (2^{4\lm-s_2} w^{2\lm+1})^{\alpha+\delta}$
and $p_2^{l_2}\leq (q/2^{s_2})^{1-\alpha}\leq (2^{4\lm-s_2} w^{2\lm+1})^{1-\alpha}$,
$p_1^{k_1}\leq w^{\alpha+\delta}$ and $p_2^{l_1}\leq w^{1-\alpha}$.
Hence, we see that $R\leq 2^{8\lm+(2\lm+1)s_1+(4\lm-s_2)\delta} w^{(1+\delta)(2\lm+1)}$
and
\begin{equation}\label{eq2992}
K\geq \frac{1}{\abs{P}\sqrt{R}}\geq \frac{1}{\abs{P}w^{(1+\delta)\left(\lm+\frac{1}{2}\right)}2^{(4+2\delta+s_1)\lm+\frac{s_1-s_2\delta}{2}}}.
\end{equation}

Combining (\ref{eq2991}) and (\ref{eq2992}) we have
\begin{equation}\label{eq2993}
\begin{split}
& \epsilon\abs{P}w^{(1+\delta)\left(\lm+\frac{1}{2}\right)}2^{(4+2\delta+s_1)\lm+\frac{s_1-s_2\delta}{2}} \\
& >1-2^{2n_1 +(4+2\delta+s_1)\lm+\frac{s_1-s_2\delta}{2}}w^{n_1+(1+\delta)\left(\lm+\frac{1}{2}\right)}\binom{n}{n_1}\left(\frac{D}{w}\right)^{n+1}G(1).
\end{split}
\end{equation}

Since $G(1)\binom{n}{n_1}=\prod_{1\leq m\leq n_1}\left(1-\frac{1}{2m}\right)<\frac{1}{8}$ for $n_1\geq 23$,
the last term of (\ref{eq2993}) is at most
\begin{equation}\label{eq2994}
\begin{split}
& 2^{2n_1 +(4+2\delta+s_1)\lm+\frac{s_1-s_2\delta}{2}}w^{n_1+(1+\delta)\left(\lm+\frac{1}{2}\right)}\left(\frac{D}{w}\right)^{n+1} \\
& \leq 2^{2n_1 +(4+2\delta+s_1)\lm+\frac{s_1-s_2\delta}{2}}w^{n_2+\delta\lm+\frac{1+\delta}{2}-(n+1)}D^{n+1} \\
& = 2^{(4+2\delta+s_1)\lm+\frac{s_1-s_2\delta}{2}}w^{\delta\lm+\frac{\delta-1}{2}}D^{\lm}\left(\frac{4D^2}{w}\right)^{n_1} \\
& \leq 2^{\left(\frac{16}{3}+2\delta+s_1\right)\lm-\frac{1}{3}}w^{\left(\delta-\frac{2}{3}\lm\right)+\frac{\delta}{2}+\frac{1}{6}}D^{\frac{7}{3}\lm-\frac{4}{3}} \\
& = \frac{w^{\frac{1}{6}+\frac{\delta}{2}}}{2^\frac{1}{3}D^\frac{4}{3}} \left(\frac{2^{16+6\delta+3s_1}D^7}{w^{2-3\delta}} \right)^{\frac{\lm}{3}} \\
& \leq\frac{1}{2},
\end{split}
\end{equation}
provided that $w^{35(2-3\delta)-(3\delta+1)/2}\geq 2^{562+210\delta+105s_1}D^{241}$, which follows from our assumption that $w\geq W\geq W_1$.
Hence, we have
\begin{equation}\label{eq2993b}
\epsilon\abs{P}w^{(1+\delta)(\lm+\frac{1}{2})}2^{(4+s_1)\lm+\frac{s_1-s_2}{2}}>\frac{1}{2}.
\end{equation}

By the property (c), we have, with the aid of Lemma 5 of \cite{Beu1},
\begin{equation}
\begin{split}
\abs{P}< & (4w)^{n_1}\binom{n}{n_1} \\
< & \frac{1}{2}\left(\frac{3}{4^{1/3}}\right)^{\frac{7}{3}\lm +2}(4w)^{\frac{2}{3}\lm +1} \\
= & 2^{-\frac{2}{9}\lm -\frac{1}{3}}3^{\frac{7}{3}\lm +2}w^{\frac{2}{3}\lm +1}.
\end{split}
\end{equation}
Now, by our assumption that $w\geq W\geq W_2$, we have
\begin{equation}
2^{\frac{4}{9}+\frac{2}{3}\delta+s_1}3^\frac{7}{3}<(4w)^{\delta_1}
\end{equation}
and therefore
\begin{equation}\label{eq2995}
\begin{split}
\abs{P}w^{(1+\delta)(\lm+\frac{1}{2})}2^{(4+2\delta+s_1)\lm+1}
< & 2^{\left(\frac{34}{9}+2\delta+s_1\right)\lm +\frac{2}{3}}3^{\frac{7}{3}\lm +2}w^{\left(\frac{5}{3}+\delta\right)\lm +\frac{3+\delta}{2}} \\
= & 2^{\frac{2}{3}}3^2 w^{\frac{3+\delta}{2}}(2^{\frac{4}{9}+s_1-\frac{4}{3}\delta}3^\frac{7}{3})^\lm(4w)^{\left(\frac{5}{3}+\delta\right)\lm} \\
\leq & 2^{4+2\delta}3^2w^{\frac{19}{6}+\frac{3}{2}\delta}(2^{\frac{4}{9}+s_1-\frac{4}{3}\delta}3^\frac{7}{3})^\lm\left(\frac{q}{w}\right)^{\frac{1}{2}\left(\frac{5}{3}+\delta\right)}\\
\leq & 2^{\frac{40}{9}+s_1+\frac{2}{3}\delta}3^{\frac{13}{3}}w^{\frac{19}{6}+\frac{3}{2}\delta}\left(\frac{q}{w}\right)^{\frac{1}{2}\left(\frac{5}{3}+\delta+\delta_1\right)}.
\end{split}
\end{equation}

Combining (\ref{eq2993b}) and (\ref{eq2995}), we have
\begin{equation}
2^{\frac{40}{9}+s_1+\frac{2}{3}\delta}3^{\frac{13}{3}}w^{\frac{19}{6}+\frac{3}{2}\delta}\left(\frac{q}{w}\right)^{\frac{1}{2}\left(\frac{5}{3}+\delta+\delta_1\right)}>\frac{1}{2\epsilon}>\frac{2q}{3D}
\end{equation}
and (\ref{eq2990}) immediately follows.
\end{proof}

\section{Arithmetic of quadratic fields and the stronger gap principle}

In this section, we shall prove a gap principle for larger solutions using some arithmetic of quadratic fields.

Let $d$ be the unique squarefree integer such that $D=B^2 d$ for some integer $B$.
We can factor $[p_1]=\gp_1 \bar\gp_1$ and $[p_2]=\gp_2 \bar\gp_2$ using some prime ideals
$\gp_1$ and $\gp_2$ in $\BQ(\sqrt{-d})$.  Moreover, if $[\alpha]=[\beta]$ in $\BQ(\sqrt{-d})$,
then $\alpha=\theta\beta$, where $\theta$ is a sixth root of unity if $d=3$,
a fourth root of unity if $d=1$ and $\pm 1$ otherwise.

Assume that $A^2+D=A^2+B^2 d=2^{2e} p_1^k p_2^l$ with $e\in \{0, 1\}$.
We must have
$[(A+B\sqrt{-d})/2^e]=\gp_1^k\gp_2^l, \bar\gp_1^{k_1}\gp_2^{l_1}, \gp_1^k\bar\gp_2^l$
or $\bar\gp_1^k\bar\gp_2^l$.
In any case we have
\begin{equation}\label{eq3998}
\left[\frac{A+B\sqrt{-d}}{A-B\sqrt{-d}}\right]=\left(\frac{\bar\gp_1}{\gp_1}\right)^{\pm k_1}\left(\frac{\bar\gp_2}{\gp_2}\right)^{\pm l_1}
\end{equation}
for some appropriate choices of signs.

We shall show a gap principle for solutions much stronger than Lemmas \ref{lm12a} and \ref{lm12b}.

\begin{lem}\label{lm31}
Let $c$ denote that constant $\sqrt{\log 2\log 3/2^{7/2}}=0.2594\cdots$.
If $x_3>x_2>x_1>10^6 D$ belong to the same set $S^{(j)}$ with $j=0$ or $3$ and $y_i=x_i^2+D$ for $i=1, 2, 3$,
then $y_3>\exp (cy_1^{1/8})$.
Furthermore, if $x_4>x_3>x_2>x_1>10^6 D$ belong to the same set $S^{(j)}$ with $j=1$ or $2$ and $y_i=x_i^2+D$ for $1\leq i\leq 4$,
then $y_4>\exp (cy_1^{1/8})$.
\end{lem}
\begin{proof}
Assume that $S^{(j)}$ has three elements $x_1<x_2<x_3$ in the case $j=0, 3$
and four elements $x_1<x_2<x_3<x_4$ in the case $j=1, 2$.
By Lemmas \ref{lm12a} and \ref{lm12b},
we have $x_4>x_3>\frac{1}{2}(y_1/4)^{3/4}>\frac{1}{2^{5/2}}x_1^\frac{3}{2}$
and we have $x_3>x_2>\frac{1}{2}(y_1/4)^{3/4}>\frac{1}{2^{5/2}}x_1^\frac{3}{2}$
in both cases respectively.
So that, setting $(X_1, X_2, X_3)=(x_1, x_2, x_3)$ in the case $j=0, 3$
and $(X_1, X_2, X_3)=(x_1, x_3, x_4)$ in the case $j=1, 2$, we have
$X_3>X_2>\frac{1}{2^{5/2}}X_1^\frac{3}{2}$ in any case.

Moreover, (\ref{eq3998}) yields that
\begin{equation}
\left[\frac{X_i+\sqrt{-D}}{X_i-\sqrt{-D}}\right]=\left(\frac{\bar\gp_1}{\gp_1}\right)^{\pm k_i}\left(\frac{\bar\gp_2}{\gp_2}\right)^{\pm l_i}
\end{equation}
for each $i=1, 2, 3$.  Hence, we obtain
\begin{equation}
\left[\frac{X_1+\sqrt{-D}}{X_1-\sqrt{-D}}\right]^{e_1}
\left[\frac{X_2+\sqrt{-D}}{X_2-\sqrt{-D}}\right]^{e_2}
\left[\frac{X_3+\sqrt{-D}}{X_3-\sqrt{-D}}\right]^{e_3}=[1],
\end{equation}
where $e_1=\pm k_2l_3\pm k_3l_2, e_2=\pm k_3l_1\pm k_1l_3, e_3=\pm k_1l_2\pm k_2l_1$
with appropriate signs are not all zero.  In other words, we have
\begin{equation}
\left(\frac{X_1+\sqrt{-D}}{X_1-\sqrt{-D}}\right)^{fe_1}
\left(\frac{X_2+\sqrt{-D}}{X_2-\sqrt{-D}}\right)^{fe_2}
\left(\frac{X_3+\sqrt{-D}}{X_3-\sqrt{-D}}\right)^{fe_3}=1,
\end{equation}
where $f=6$ if $d=3$, $4$ if $d=1$ and $2$ otherwise.
This implies that $\Lm=e_1\arg(X_1\pm\sqrt{-D})+e_2\arg(X_2\pm\sqrt{-D})+e_3\arg(X_3\pm\sqrt{-D})$
must be a multiple of $2\pi/f$.

If $\Lm\neq 0$, then we see that
$\left(\frac{\abs{e_1}}{X_1}+\frac{\abs{e_2}}{X_2}+\frac{\abs{e_3}}{X_3}\right)\sqrt{D}>\abs{\Lm}\geq 2\pi/f$
and therefore $2.01fKL\sqrt{D}\geq 2X_1\pi$ and $1.92KL\sqrt{D}>X_1$.  Since $X_1=x_1>10^6 D$,
we have $1.92KL>2\sqrt{X_1}>(X_1^2+D)^{1/4}=y^{1/4}$.

Assume that $\Lm=0$.
If $e_1=0$, then we must have $\Lm=e_2\arg(X_2\pm\sqrt{-D})+e_3\arg(X_3\pm\sqrt{-D})=0$
and $(X_2^2+D)^{e_2}=(X_3^2+D)^{e_3}$.  Hence, we must have
$\abs{e_2}>\abs{e_3}>0$ and $\abs{\arg(X_2\pm\sqrt{-D})}>\abs{\arg(X_3\pm\sqrt{-D})}>0$ from $X_3>X_2$
and $\Lm\neq 0$, which is a contradiction.
Thus $e_1$ cannot be zero.
The triangle inequality immediately gives that
$\abs{\arg(X_1\pm\sqrt{-D})}\leq \abs{e_2\arg(X_2\pm\sqrt{-D})}+\abs{e_3\arg(X_3\pm\sqrt{-D})}$
and therefore
\begin{equation}
\frac{1}{\sqrt{X_1^2+D}}<\frac{e_2}{X_2}+\frac{e_3}{X_3}<\frac{2KL}{X_2}<\frac{8\sqrt{2}KL}{(X_1^2+D)^\frac{3}{4}}.
\end{equation}
Thus we obtain $8\sqrt{2}KL>(X_1^2+D)^{1/4}=y_1^{1/4}$.

Hence, in any case we have $8\sqrt{2}KL>y_1^{1/4}$
and $\max\{ K\log p_1, L\log p_2\}>y_1^{1/8}\sqrt{\log p_1\log p_2/(8\sqrt{2})}$.
Thus, we conclude that
\begin{equation}\label{eq3999}
X_3^2+D\geq \max\{p_1^K, p_2^L\}>\exp \left(y_1^{\frac{1}{8}}\sqrt{\frac{\log p_1\log p_2}{8\sqrt{2}}}\right)
\geq \exp \left(cy_1^\frac{1}{8}\right),
\end{equation}
proving the Lemma.
\end{proof}

\section{Proof of the Theorem}

We set $\delta_1=0.04377667$.  We shall begin by proving (i).

Let $y_1=x_1^2+D$ be the smallest solution in a given class $S^{(j)}(W, \infty; u, v)$
and $y_2=x_2^2+D$ be the third or fourth smallest one in this class for $j=0, 3$ or $j=1, 2$, respectively.
Lemma \ref{lm21} with $\delta=1/4$ gives that
\begin{equation}
y_2<\max \left\{4^{70} y_1^{71}, (2^{\frac{101}{18}}3^{\frac{16}{3}}D y_1^{\frac{31}{12}-\frac{\delta_1}{2}})^{1/\left(\frac{1}{24}-\frac{\delta_1}{2}\right)}\right\}.
\end{equation}
But Lemma \ref{lm31} immediately yields that $y_2>\exp (cy_1^{1/8})$.
We observe that these two inequalities are incompatible for $y_1\geq W=\max\{W_1, W_2\}$.
Hence, we see that $\# S^{(j)}(W, \infty; u, v)\leq 2$ for each $j, u, v$ for $j=0, 3$
and $\# S^{(j)}(W, \infty; u, v)\leq 3$ for each $j, u, v$ for $j=1, 2$.
Combining these estimates, we obtain $\# S(0, 1, W, \infty)\leq 30$
after the easy observation that $S(0, 1, W, \infty; 0, 0)$ must be empty since $W>D^2$.
This proves (i).

Now we shift our concern to smaller solutions.
Let $f(y)=y^{3/2}/2^{5/2}, g(y)=\exp (cy_1^{1/8})$ and $f^{(m)}$ be the $m$-th iteration of $f$.
$y_1=x^2+D=2^s p_1^k p_2^l$ denotes the smallest solution.  We have the following three cases.

Case 1. $D\geq Y$ or $y_1\geq Y$.

If $D\geq Y$, then $W=W_1<g(f^{(3)}(D))\leq g(f^{(3)}(y_1))$.
If $D\leq Y-1$ and $y_1\geq Y$, then we have that $W=W_2<g(f^{(3)}(Y))\leq g(f^{(3)}(y_1))$.
Hence, we always have $W\leq g(f^{(3)}(y_1))$ in Case 1
and therefore, using Lemmas \ref{lm12a} and \ref{lm12b},
we obtain $\# S^{(j)}(y_1, W)\leq 9$ if $j=0, 3$ and $\# S^{(j)}(y_1, W)\leq 5$ if $j=1, 2$.
So that, $\#S(0, 1, 0, \infty)\leq 30+28=58$.
This proves (ii).

Case 2. $D, y_1, p_2\leq Y-1$.

Let $W_3=f^{(2)}(Y)=3545401233665.83\cdots$.
Since $y_1\leq Y-1$, then $D\leq y_1\leq Y-1$ and $p_1\leq y_1\leq Y-1$.
A computer search revealed that
$\#S(0, 1, 0, W_3)\leq 13$ for any $D, p_1, p_2\leq Y-1$.
Since $W<g(f^{(3)}(Y))=g(f(W_3))$,
from Lemmas \ref{lm12a} and \ref{lm12b} we see that
$\# S^{(j)}(W_3, W)\leq 5$ if $j=0, 3$
and $\# S^{(j)}(W_3, W)\leq 3$ if $j=1, 2$.
This proves (iii).

Case 3. $D, y_1\leq Y-1$ and $p_2\geq Y$.

If $x^2+D=2^s p_1^k p_2^l\leq Y-1$, then,
since $p_2\geq Y$, we must have $x^2+D=2^s p_1^k$,
which has at most five solutions from the results mentioned in the introduction.
The number of the other solutions can be bounded as in Case 2
and we obtain (iv).  This completes the proof of the Theorem.

{}

{\small Tomohiro Yamada}\\
{\small Center for Japanese language and culture\\Osaka University\\562-8558\\8-1-1, Aomatanihigashi, Minoo, Osaka\\Japan}\\
{\small e-mail: \protect\normalfont\ttfamily{tyamada1093@gmail.com}}
\end{document}